\documentclass[11pt]{article}
\usepackage[utf8]{inputenc}
\usepackage{lmodern}
\usepackage{subfiles}
\usepackage{enumitem}
\setenumerate{topsep=6pt,ref={\normalfont(\arabic*)},label={\normalfont(\arabic*)}, itemsep=0pt} 
        
\usepackage{amsfonts}
\usepackage{amsthm}
\usepackage{amsmath}
\usepackage{amssymb}
\usepackage{enumitem}
\usepackage{amscd}
\usepackage{mathrsfs}
\usepackage{mathtools}
\usepackage{bbm}
\usepackage{esint}

\usepackage[margin=3cm]{geometry}
\usepackage{setspace}
\usepackage{indentfirst}
\usepackage{graphicx}
\usepackage{graphics}
\usepackage{lscape}
\usepackage{pgf,tikz}
\usepackage{tikz-cd}
\usepackage{color}
\usepackage{pict2e}
\usepackage{epic}
\usepackage{epstopdf}
\usepackage{titlesec, titlefoot}
	\titleformat{\section}[block]{\Large\bfseries\filcenter}{\thesection}{1em}{}
\usepackage{commath}
\usepackage{float}
\usepackage{caption}
\usepackage{etoolbox}
\usepackage[affil-it]{authblk}
\usepackage{combelow}

\usepackage[hidelinks, bookmarksdepth=3,citecolor=blue,colorlinks]{hyperref}

\graphicspath{{./Pictures/}}
\allowdisplaybreaks

\expandafter\def\expandafter\normalsize\expandafter{%
    \normalsize
    \setlength\abovedisplayskip{6pt}
    \setlength\belowdisplayskip{6pt}
    \setlength\abovedisplayshortskip{6pt}
    \setlength\belowdisplayshortskip{6pt}
}

\theoremstyle{plain}

\renewcommand*\thesection{\arabic{section}}
\numberwithin{equation}{section}

\newtheorem{theorem}{Theorem}[section]
\newtheorem{lemma}[theorem]{Lemma}
\newtheorem*{lemma*}{Lemma}
\newtheorem{proposition}[theorem]{Proposition}

\theoremstyle{definition}
\newtheorem{definition}[theorem]{Definition}

\newtheorem*{remark*}{Remark}

\expandafter\let\expandafter\oldproof\csname\string\proof\endcsname
\let\oldendproof\endproof
\renewenvironment{proof}[1][\proofname]{%
  \oldproof[\upshape \bfseries #1]%
}{\oldendproof}

\makeatletter
\def\@makechapterhead#1{%
  \vspace*{50\p@}%
  {\parindent \z@ \raggedright \normalfont
    \interlinepenalty\@M
    \Huge\bfseries  \thechapter.\quad #1\par\nobreak
    \vskip 40\p@
  }}
\makeatother

\renewcommand{\Re}{\operatorname{Re}}
\renewcommand{\Im}{\operatorname{Im}}

\def \R {\mathbb{R}}

\def \C{\mathbb{C}}

\def \D{\textup{D}}

\def \T{\mathbb{T}}
\def \e{\varepsilon}
\def \d{\textup{d}}

\def \p{\partial}
\def \mc{\mathcal}
\def \mb{\mathbb}

\def \loc{\textup{loc}}

\newcommand{\ol}{\overline}

\DeclareMathOperator{\co}{co}
\DeclareMathOperator{\dist}{dist}

\DeclareMathOperator{\supp}{supp}

\DeclareMathOperator{\Lip}{Lip}

\begin{document}

	\title{\textbf{A singular minimizer of a smooth strongly polyconvex functional in two dimensions}}
		
	\author[1]{{\Large Andr\'e Guerra}}
	\author[2]{{\Large Riccardo Tione}}
		
	\affil[1]{\small Department of Pure Mathematics and Mathematical Statistics,  University of Cambridge,\protect\\  Wilberforce Rd, Cambridge CB3 0WB, UK
	\protect\\
	{\tt{adblg2@cam.ac.uk}} \vspace{1em} \ }
		
	\affil[2]{\small Dipartimento di Matematica, Universit\`a degli Studi di Torino,
Via Carlo Alberto 10, 10123 Torino, Italy 
	\protect\\
	{\tt{riccardo.tione@unito.it}}  }

	\date{}
	
	\maketitle
	
	\begin{abstract}
	We construct a Lipschitz, non-$C^1$ map $u\colon \R^2\to \R^6$ which is the unique global minimizer of a smooth, strongly polyconvex functional. Such an example is necessarily non-homogeneous.
	\end{abstract}

\section{Introduction}

	We study minima of vectorial variational integrals
	$$\mb E[u] \equiv \int_\Omega F(\D u) \d x, \qquad u\colon \Omega\to \R^{m},\qquad m\geq 2,$$	
	where $F\in C^\infty(\R^{m\times n})$ and $\Omega\subset \R^n$ is a domain.  We will assume throughout that 
	$$|F''(A)|\leq C\quad \text{for all }A\in \R^{m \times n}.$$
	Our focus is on the case where $F$ is strongly polyconvex:
	
	\begin{definition}\label{def:pc}
		 The integrand $F\colon \R^{m\times n}\to\R$ is
		 \textit{polyconvex} if there is a convex function
		 $f\colon \R^{\tau(m,n)}\to\R$ such that
		 $F(A)=f(M(A))$, where $M(A)\in\R^{\tau(m,n)}$ is the vector of
		 all minors of $A$ of orders $1,\ldots,\min\{m,n\}$, and $\tau(m,n)\equiv
		 \sum_{k=1}^{\min\{m,n\}}\binom{m}{k}\binom{n}{k}.$ 
		 
		 $F$ is
		 \textit{strongly polyconvex} if there is $\e>0$ such that
		 $F-\e|\cdot|^2$ is polyconvex.
	\end{definition}
	
	Polyconvexity is a natural condition for the existence of minimizers in vectorial problems \cite{Dacorogna2007} and it holds in most of  the energies of interest in Nonlinear Elasticity \cite{Ball1977}, Geometric Function Theory \cite{Iwaniec2001}, and Geometric Measure Theory \cite{DeLellis2019}.
	
	Even if $F$ is convex, in general minimizers are not smooth as soon as $n\geq 3$ \cite{Mooney2016}, and hence the best one can hope for is partial regularity. In this case, the strongest result available \cite[\S 4.4]{Ambrosio2018} asserts that there is a closed singular set $\Sigma_u$ such that $u \in C^\infty(\mb B^n\setminus \Sigma_u),$ and
\begin{equation}
\label{eq:convexbound}
 \dim_H(\Sigma_u)\leq n-2-\delta, \qquad \delta=\delta(u,F)>0.
\end{equation}
	The example in \cite{Mooney2016}, together with its predecessors \cite{Hao1996,Necas1977,SverakYan1999,SverakYan2002} in higher dimensions, shows that in general $\delta\leq 1$. We also emphasize that, in all of these examples, the map $u$ is homogeneous.
	
	By contrast, much less is known for polyconvex functionals.\ For Lipschitz minimizers,  we again have $u \in C^\infty(\mb B^n\setminus \Sigma_u),$ this time with
\begin{align}
\label{eq:polybound}
\Sigma_u \text{ is uniformly porous }
\qquad \implies \qquad \dim_H(\Sigma_u) \leq n-\delta, \quad \text{ } \delta=\delta(u,F)>0,
\end{align}
as shown in \cite{Kristensen2007}.\ In particular, the methods of \cite{Kristensen2007} necessarily yield $\delta\leq 1$, since hyperplanes are uniformly porous.\ The estimate on the dimension of the singular set was very recently extended to non-Lipschitz minimizers in \cite{DeFilippis2026}.\ Classical works on non-Lipschitz minimizers include the original partial regularity results \cite{Evans1986}, but see also \cite{Acerbi1987,Carozza2007,Esposito2001,Fusco1991,Gmeineder2022,Hamburger2003,Schmidt2009} for results under more general growth conditions. Stronger results are available for particular functionals, e.g.\ for the area functional \cite{Giaquinta2012,Hirsch2023} or functionals with some symmetries \cite{GuerraTione2024,GuerraTione2026}.

It has remained a long-standing problem to improve \eqref{eq:polybound}, or to construct examples showing otherwise.  We briefly explain the difficulties underlying this problem:
\begin{enumerate}
\item the bound \eqref{eq:convexbound} is by now standard, and is proved by differentiating the Euler--Lagrange system. By contrast, in the polyconvex setting the Euler--Lagrange system admits solutions which are singular everywhere \cite{Muller2003,Szekelyhidi2004};
\item no examples of singular minimizers for polyconvex, non-convex functionals were known. Since, by \eqref{eq:convexbound}, minimizers of convex functionals are regular when $n=2$, in this case no examples of singular minimizers for polyconvex functionals were known;
\item to construct singular Lipschitz minimizers it is natural to look for 1-homogeneous examples, and indeed, as mentioned above, all known examples of singular minimizers are homogeneous. Yet, when $n=2$, any 1-homogeneous solution of the Euler--Lagrange system of a polyconvex functional is necessarily smooth \cite{Phillips2002}.
\end{enumerate}

We overcome the above difficulties and construct for the first time a singular minimizer for a non-convex, strongly polyconvex functional, and we do so in the lowest possible dimension $n=2$. This shows that the best general upper bound one can hope for is $\dim_H(\Sigma_u)\leq n-2$.

To describe our example,  let $u\colon \C\cong \R^2\to \R^{2N}\cong \C^N$, $u=(u^1, \dots, u^N)$, be defined by
\begin{equation}
\label{eq:defu}
u^j(z) \equiv r^{1+ i \omega_j} e^{i k_j\theta}, \qquad j=1,\dots, N.
\end{equation}
The possibility of using maps of this form to construct a counterexample was suggested in a beautiful lecture by \v Sver\'ak \cite{Sverak2025}, where this ansatz is attributed to B. Kirchheim and V. \v Sver\'ak; see also \cite[p.\ 152]{Muller1999a}. We will make what is, in some sense, the simplest possible choice of frequencies: 
\begin{equation}
\label{eq:choicek}
N=3, \qquad (k_1, \omega_1)=(1,1), \qquad (k_2,\omega_2)=(1,2), \qquad (k_3,\omega_3)=(3,3).
\end{equation}

Our main theorem is the following:
\begin{theorem}\label{thm:main}
Let $u\in W^{1,\infty}(\mb B^2,\R^6)$ be the map specified by \eqref{eq:defu}--\eqref{eq:choicek}. There is a smooth, strongly polyconvex integrand $F\colon \R^{6\times 2}\to \R$, with $|F''|\leq C$, such that
$$\int_{\mb B^2} F(\D u) \, \d x \leq \int_{\mb B^2} F(\D v) \, \d x \qquad \text{for all } v\in u+W^{1,2}_0(\mb B^2,\R^6),$$
with equality if and only if $u=v$. 
\end{theorem}

Observe that, differently from the convex setting, in the polyconvex case uniqueness of minimizers is, in general, false \cite{Lawson1977,Spadaro2009}.\ Those works provide examples of smooth, global minimizers to polyconvex energies, while here $u$ is the \emph{unique} global minimizer of the energy, thus showing that the singular behaviour of the map at hand is unrelated to the lack of uniqueness.

\medskip

In the rest of the introduction, we describe our proof strategy.\ Let us first note a simple calculation: for a map $v(z)= r^{1+ i \omega} e^{i k \theta}$, we have
\begin{equation}
\label{eq:complexders}
\p_z v = \frac{1 + k + i \omega}{2} r^{i \omega} e^{i (k-1)\theta}, \qquad \p_{\bar z} v = \frac{1-k+i\omega}{2} r^{i \omega} e^{i (k+1)\theta}.
\end{equation}
In particular, $|\D u|^2$ is constant so $u$ is Lipschitz. Moreover, the oscillatory factors $r^{i\omega_j}$ show that $\D u$ has no limit at the origin, thus $u\notin C^1$.\
To see why \eqref{eq:choicek} is natural, we note that it corresponds to the smallest frequencies for which there are many linear relations between different minors of $\D u$.\ To see this, we introduce the following notation: if $A \simeq (a,c)$ and $B \simeq (b,d)$, where $a,b$ ($c,d$) are the conformal (anti-conformal) parts of $A,B$ respectively, let
\begin{equation}\label{eq:wedge}
A\wedge B \equiv a d-bc \in \C,
\end{equation}
which corresponds to the determinant of the complex-valued matrix with rows $(a,c)$ and $(b,d)$.\ Through \eqref{eq:complexders}, we see that the choices \eqref{eq:defu}-\eqref{eq:choicek} lead to the six relations which completely determine the range of $\D u$ (Lemma \ref{lem:zero-set}).\ For example, one of these relations is:
\[
\D u^1 \wedge \D u^2  = \frac{i}{2}r^{3 i} e^{2 i\theta}  = \frac{3+4i}{25} \p_z u^3.
\]
Our approach is precisely to analyse this range, and we set
\begin{equation}
\label{eq:defK}
K\equiv \D u(\mb B^2\backslash \{0\}).
\end{equation}
This set is easily seen to be a torus $\mb T^2$ smoothly embedded in a sphere in $\R^{6\times 2}$ (Lemma \ref{lemma:defK}).

A  na\"ive but natural candidate for the integrand in Theorem \ref{thm:main} is $F=\dist^2_K$, and we refer the reader to \cite{Dolzmann2013}  for the study of minimization problems for the corresponding functional. However, $\dist^2_K$ is not polyconvex (it is not even convex separately in each matrix variable), since $K$ is not convex \cite[Theorem 4]{Dolzmann2000}. Nonetheless, one can still hope to build a strongly polyconvex integrand $F$ which has a well on $K$, i.e.\ $F=\dist^2_K$ near $K$. It is relatively clear what local conditions are required to extend a given integrand, defined on a ball, to a global polyconvex integrand \cite{Kristensen2000}. However, again due to the non-convexity of $K$, in our case local conditions are insufficient and a more global argument is needed. In Section \ref{sec:squared-distance-extension} we provide simple, general conditions under which it is possible to find such an extension for $\dist^2_K$, and we verify that they hold in our case.

Finally, to prove that $u$ is the unique global minimizer we need a separate argument concerning solutions of the differential inclusion $\D v \in K$, and we prove:

\begin{theorem}\label{thm:rigid}
Let $\Omega\subset \C$ be connected and let $v\in W^{1,\infty}_\loc(\Omega,\C^3)$ solve $\D v\in K$ a.e.\ in $\Omega$. Then either $v$ is affine or else there is $\lambda \in \mb S^1,z_0\in \C$ such that $\D v = \D(S_\lambda u(\cdot-z_0))$, where
$$S_\lambda (u^1,u^2,u^3) \equiv (\lambda u^1, \lambda^2 u^2, \lambda^3 u^3).$$
\end{theorem}

To prove Theorem \ref{thm:rigid} we find, after an appropriate quasiconformal change of variables, that the last component of solutions to $\D v\in K$ solves a degenerate elliptic differential inclusion in the sense of \cite{Lamy2024}. We can then apply their results, see also \cite{Lamy2026} for the completely general version, to conclude.

{\small
\bigskip
\textbf{Acknowledgements.}
 AG acknowledges the support of the Royal Society through a Newton International Fellowship. 
 RT was funded by the Italian Ministry for University and Research (MUR), through FIS3 Starting Grant “GEMS” CUP: D53C25002540001 (Finanziata con il contributo del Ministero dell’Università e della ricerca ai sensi del D.D. n. 1802 del 21-11-2024 - BANDO FIS 3, Grant number FIS-2024-02219). 

\bigskip
\textbf{AI disclosure.}
ChatGPT 5.5-Plus was used for finding the complete system of quadratic polyaffine relations determining the set $K$, in Lemma \ref{lem:zero-set}, and for verifying that Theorem \ref{thm:rigid} holds for smooth solutions.\ ChatGPT 5.5-Plus and Claude Opus 5.5 were used to revise the manuscript.
}

\section{The range of the differential of the map}
\label{sec:range}

This section is devoted to analyze the set $K \subset \R^{6 \times 2}$ introduced in \eqref{eq:defK}. To this end, we start by introducing some useful notation, Subsection \ref{subsec:not}.\ Next, in Subsection \ref{subsec:poly}, we write the system of polyaffine relations mentioned in the introduction, and we study its nondegeneracy in the last Subsection \ref{subsec:nondegpoly}.\ These results will be instrumental for the rest of the proof.\

\subsection{Notation}\label{subsec:not}

It is convenient to identify $\R^2\cong \C$ and $\R^{6\times 2}\cong \C^{3\times 2}$ using complex coordinates.\ If
\[
M = \left[\begin{array}{cc} a & b \\ c & d\end{array}\right],
\]
its conformal and anti-conformal parts are, respectively:
\[
[M]_{\mathcal{H}} \equiv  \frac{1}{2}[(a + d) + i(c-b)] \quad \text{ and } \quad [M]_{\overline{\mathcal{H}}} \equiv  \frac{1}{2}[(a - d) + i(c+ b)].
\]
Direct computations show that
\begin{equation}\label{eq:alg}
	\det(M) = |[M]_{\mathcal{H}}|^2 - |[M]_{\overline{\mathcal{H}}}|^2, \quad |M|^2 = 2|[M]_{\mathcal{H}}|^2 + 2|[M]_{\overline{\mathcal{H}}}|^2.
\end{equation}

Therefore, letting $A \in \R^{6\times 2}$ and subdividing it in three $2\times 2$ blocks $A_1,A_2,A_3 \in \R^{2\times 2}$ which contain the first two, second two and last two rows respectively, we write
\begin{equation}
	\label{eq:defA}
	A = \begin{bmatrix}
		A_1\\
		A_2\\
		A_3
	\end{bmatrix} \cong
	\begin{bmatrix}
		[A_1]_{\mathcal{H}} & [A_1]_{\overline{\mathcal{{H}}}}\\
		[A_2]_{\mathcal{H}} & [A_2]_{\overline{\mathcal{{H}}}}\\
		[A_3]_{\mathcal{H}} & [A_3]_{\overline{\mathcal{{H}}}}
	\end{bmatrix} = \begin{bmatrix}
	a_1 & b_1\\
	a_2 & b_2\\
	a_3 & b_3
	\end{bmatrix}
	\in\C^{3\times2}.
\end{equation}
By defining the Wirtinger derivatives of a differentiable map $f: \C \to \C$ as $\partial_z f := [Df]_{\mathcal{H}}$ and $\partial_{\bar z}f := [Df]_{\overline{\mathcal{H}}}$, \eqref{eq:defA} yields:

$$
\D u \cong 
\begin{bmatrix}
\p_z u^1 & \p_{\bar z} u^1\\
\p_z u^2 & \p_{\bar z} u^2\\
\p_z u^3 & \p_{\bar z} u^3\\
\end{bmatrix}
$$
The next lemma follows by a straightforward calculation from \eqref{eq:complexders}, which we omit.

\begin{lemma}\label{lemma:defK}
Consider the constants
\[
c_1=1+\frac i2,\quad d_1=\frac i2,\quad
c_2=1+i,\quad d_2=i,\quad 
c_3=\frac{4+3i}{2},\quad
d_3=\frac{-2+3i}{2}.
\]
Then, in the notation \eqref{eq:defK}, we have
$$K = \varphi(\mb T^2), \qquad 
\varphi(\xi,\eta)\equiv 
\begin{bmatrix}
c_1 \xi & d_1 \eta \\
c_2 \xi^2 & d_2 \xi \eta \\
c_3 \xi^2 \eta & d_3 \xi \eta^2
\end{bmatrix}.
$$
Moreover $\varphi\colon \T^2 \to \C^{3\times 2}$ is an embedding, thus $K$ is a smooth, embedded torus in $\C^{3\times 2}$.
\end{lemma}

Let us note a simple invariance of $K$. For $\lambda,\mu\in\mb S^1$, define the  map $\mc T_{\lambda,\mu}\colon \C^{3\times 2}\to \C^{3\times 2}$:
\begin{equation}
\label{eq:torus-action}
\mathcal T_{\lambda,\mu}(A)\equiv
\begin{bmatrix}
\lambda a_1&\mu b_1\\
\lambda^2a_2&\lambda\mu b_2\\
\lambda^2\mu a_3&\lambda\mu^2b_3
\end{bmatrix},
\end{equation}
which  is real-linear orthogonal. Directly from the definition of $\varphi$, we have
\begin{equation}
\label{eq:torus-equivariance}
\mathcal T_{\lambda,\mu}\bigl(\varphi(\xi,\eta)\bigr)
=\varphi(\lambda\xi,\mu\eta).
\end{equation}
In particular, $\mathcal T_{\lambda,\mu}(K)=K$.

\subsection{The polyaffine relations}\label{subsec:poly}

The goal of this subsection is to show that $K$ can be characterized through a system of polyaffine relations.\ We say that $F\colon \R^{m\times n}\to \R$ is polyaffine if $\pm F$ are polyconvex.\ There are several equivalent definitions of polyaffinity, that we collect in the following:
\begin{lemma}\label{lem:polyaff}
	$F\colon \R^{m\times n}\to \R$ is polyaffine if and only if $F$ is an affine combination of minors.\ If $F$ is a quadratic form, then $F$ is polyaffine if and only if $F(R) = 0$ for every $R \in \R^{m\times n}$ with rank one.
\end{lemma}
\begin{proof}
	The general equivalence can be found in \cite{Dacorogna2007}, while the characterization of quadratic $F$ can be deduced following the same proof as \cite[Lemma 4.2]{Guerra2018}.
\end{proof}

\medskip

 Operation \eqref{eq:wedge} allows to define all the polyaffine relations we require: setting $u^\ell = \overline{u^{\ell - 3}}$ for $\ell = 4,5,6$, we claim that these polyaffine relations are given by
\begin{equation}\label{eq:ab}
Du^a\wedge Du^b = c^1_{ab} + c^2_{ab}\partial_{z}u^{r} +c^3_{ab}\partial_{\bar z} u^s \text{ a.e.}, \quad \text{for some }r(a,b),s(a,b) \in \{1,\dots,6\},c^k_{ab} \in \C,
\end{equation}
for suitable choices of $a$ and $b$.\ There are $36$ such wedge products, but, due to the anti-symmetry of the operation, only $15$ give nontrivial information, and we will only need 6 of them.\ To simplify the notation, set:
\[
J_j(A)\equiv \det A_j \overset{\eqref{eq:alg}}{=} |a_j|^2-|b_j|^2,\quad 
H_{jk}(A)\equiv A_j\wedge A_k = a_jb_k-b_ja_k,\quad 
L_{jk}(A)\equiv a_k\overline{a_j}-b_k\overline{b_j}.
\]
While $H_{jk}$ is $A_j \wedge A_k$ and hence, once evaluated on $\D u$, it is the left-hand side of \eqref{eq:ab} for $a,b \in \{1,2,3\}$, $L_{jk}$ corresponds to considering a wedge product $\D u^a \wedge \D u^b$ for $a \in \{1,2,3\}$, $b \in \{4,5,6\}$.\ Finally, notice that $J_j(A) = L_{jj}(A)$.\ $J_j$, $H_{jk}$ and $L_{jk}$ all vanish on rank-one matrices in $\R^{6\times 2}$, hence they are polyaffine.\ Thus, we define the vector of the minor relations we will need:
\[
        \Phi=(\Phi_1,\Phi_3,\Phi_H,\Phi_{12},\Phi_{13},\Phi_{23})\colon \C^{3\times 2}\to \R^2 \times \C^4\cong \R^{10},
\]
through
\begin{gather*}
\Phi_1=J_1-1,\qquad
\Phi_3=J_3-3,\qquad
\Phi_H=H_{12}-\frac{3+4i}{25}a_3,\\
\Phi_{12}=L_{12}-a_1,\qquad 
\Phi_{13}=L_{13}+2i b_2,\qquad
\Phi_{23}=L_{23}+(3+4i)b_1 .
\end{gather*}

\begin{lemma}\label{lem:zero-set}
The above relations characterize $K$, i.e.\
$ \Phi^{-1}(0)=K.$
\end{lemma}

\begin{proof}
Let $A$ be as in \eqref{eq:defA}.
Since all $c_j$ and $d_j$ are non-zero, let
\[
        a_j=c_j\xi_j,\qquad b_j=d_j\eta_j,\qquad j=1,2,3.
\]
The equations $\Phi(A)=0$ are equivalent to
\begin{gather}
\label{eq:J}
\tag{$J$} 5|\xi_1|^2-|\eta_1|^2=4,\qquad
25|\xi_3|^2-13|\eta_3|^2=12,\\
\label{eq:H12}
\tag{$H_{12}$} \xi_3=(-1-i)\eta_1\xi_2+(2+i)\xi_1\eta_2,\\
\label{eq:L12}
\tag{$L_{12}$} (3+i)\overline{\xi_1} \xi_2-\overline{\eta_1} \eta_2=(2+i)\xi_1,\\
\label{eq:L13}
\tag{$L_{13}$} (11+2i)\overline{\xi_1} \xi_3-(3+2i)\overline{\eta_1} \eta_3=8\eta_2,\\
\label{eq:L23}
\tag{$L_{23}$} (7-i)\overline{\xi_2} \xi_3-(3+2i)\overline{\eta_2} \eta_3=(4-3i)\eta_1.
\end{gather}
It is easy to see that the system \eqref{eq:J}--\eqref{eq:L23} is equivariant under the action \eqref{eq:torus-action}.\ Furthermore, $\xi_\ell = 1, \eta_\ell = 1$ for all $\ell \in \{1,2,3\}$ solve system \eqref{eq:J}--\eqref{eq:L23}, and hence $K\subset \Phi^{-1}(0)$, since $K = \{\mathcal{T}_{\lambda,\mu}(\varphi(1,1)): \lambda,\mu \in \mathbb{S}^1\}$, using the notation of Lemma \ref{lemma:defK} and the subsequent definitions.\ We show the converse: to this end, let $A\in\Phi^{-1}(0)$.\ Our goal is to show that $A \in K$.

\medskip

The first $J$-equation gives $|\xi_1|^2\ge 4/5$.\ Using the equivariance of the system under the action \eqref{eq:torus-action}, we may assume $\xi_1 >0, \eta_1 \ge 0$, so that our goal becomes $A = \varphi(1,1)$.\ \eqref{eq:L12} gives
\begin{equation}\label{e:xi2}
\xi_2=\frac{(2+i)\xi_1 + \eta_1 \eta_2}{(3+i)\xi_1} = \frac{2+i}{3 + i} +\frac{\eta_1}{(3+i)\xi_1} \eta_2.
\end{equation}
This can be substituted inside \eqref{eq:H12} to get, after using the first $J$-equation:
\begin{equation}\label{e:xi3}
	\begin{split}
		\xi_3= \frac{-(1+3i)\eta_1\xi_1 + (4+4i)\eta_2}{(3+i)\xi_1} =  -\frac{1+3i}{3+i}\eta_1 +  \frac{4+4i}{(3+i)\xi_1}\eta_2
	\end{split}	
\end{equation}
Now from \eqref{eq:L13} we obtain
\begin{equation}\label{e:eta3pre}
(3+2i)\eta_1 \eta_3=(11+2i)\xi_1 \xi_3-8\eta_2 \overset{\eqref{e:xi3}}{=} -5\frac{1+7i}{3+i}\xi_1\eta_1 + 4\frac{3+11i}{3+i}\eta_2.
\end{equation}
If we had $\eta_1 = 0$, then \eqref{e:eta3pre} would give us $\eta_2 = 0$, and from \eqref{e:xi3} we would deduce $\xi_3 = 0$, in contradiction with the second $J$-equation.\ Hence, from now on we can assume
\begin{equation}\label{e:sign}
\xi_1>0,\qquad \eta_1> 0.
\end{equation}
Thus, we can divide by $(3+2i)\eta_1$ in \eqref{e:eta3pre} to obtain
\begin{equation}\label{e:eta3}
\eta_3=-5\frac{1+7i}{7+9i}\xi_1+ 4\frac{3+11i}{7+9i}\frac{\eta_2}{ \eta_1}.
\end{equation}
Now we need to use the last equation, \eqref{eq:L23}.\ We start by computing

\begin{align*}
	(7-i)\overline{\xi_2} \xi_3 &\overset{\eqref{e:xi2}-\eqref{e:xi3}}{=}-(4+3i)\eta_1 - (1+2i)\frac{\eta_1^2}{\xi_1}\overline{\eta_2} +\frac{4(11+2i)}{5\xi_1}\eta_2 + \frac{4(4+3i)\eta_1}{5\xi_1^2}|\eta_2|^2,\\
	  -(3+2i)\overline{\eta_2} \eta_3 &\overset{\eqref{e:eta3}}{=} 5(1+2i)\xi_1\overline{\eta_2} -4(2+3i)\frac{|\eta_2|^2}{ \eta_1}.
\end{align*}
Exploiting once again the first $J$-equation, we have
\[
(7-i)\overline{\xi_2} \xi_3-(3+2i)\overline{\eta_2} \eta_3 = -(4+3i)\eta_1 + 4(1+2i)\frac{\overline{\eta_2}}{\xi_1} +\frac{4(11+2i)}{5\xi_1}\eta_2 + \frac{4(4+3i)\eta_1}{5\xi_1^2}|\eta_2|^2 -4(2+3i)\frac{|\eta_2|^2}{ \eta_1}.
\]
We make use of \eqref{eq:L23}, where we bring the first addendum of the previous sum on the right-hand side and multiply by $5\xi_1^2\eta_1/4$, to get:
\begin{equation}\label{e:L23mod}
 5(1+2i)\xi_1\eta_1\overline{\eta_2} +(11+2i)\xi_1\eta_1\eta_2 + (4+3i)\eta_1^2|\eta_2|^2 -5(2+3i)\xi_1^2|\eta_2|^2 = 10\eta_1^2\xi_1^2.
\end{equation}
\eqref{e:L23mod} consists of two nonlinear equations in $\eta_2 = x + iy$.\ We can make it linear by considering it as a system in three variables, $x,y$ and $n = |\eta_2|^2$.\ Coupling these with the second $J$-equation, we find an invertible system of three linear equations in three variables: using that $n = x^2 + y^2$ will then give us the required rigidity to conclude the proof.\ We start by taking respectively the real and the imaginary part of \eqref{e:L23mod}:
\begin{align}
	4(2x+y)\xi_1\eta_1 + 2\eta_1^2n -5\xi_1^2n = 5\eta_1^2\xi_1^2,\label{e:L23re}\\
    2(2x+y)\xi_1\eta_1 + \eta_1^2n -5\xi_1^2n = 0\label{e:L23im}
\end{align}
Subtracting \eqref{e:L23re} from twice \eqref{e:L23im}, we find that
\begin{equation}\label{e:n}
n = |\eta_2|^2 = \eta_1^2.
\end{equation}
Using this and the first $J$-equation, after subtracting \eqref{e:L23im} from \eqref{e:L23re} we get 
\begin{equation}\label{e:xy}
	\xi_1(2x+y) = 2\eta_1.
\end{equation}
Finally, we use the second $J$-equation.\ To this end, we compute
\begin{align*}
25|\xi_3|^2 &\overset{\eqref{e:xi3}-\eqref{e:n}}{=} 25\eta_1^2 + 80\frac{\eta_1^2}{\xi_1^2} - 40\frac{\eta_1}{\xi_1}(2x+y) \overset{\eqref{e:xy}}{=} 25\eta_1^2,\\
13|\eta_3|^2 &\overset{\eqref{e:eta3}-\eqref{e:n}}{=} 125\xi_1^2 + 208 - 40(8x+y)\frac{\xi_1}{\eta_1} \overset{\eqref{e:xy}}{=} 125\xi_1^2 + 128 - 240x\frac{\xi_1}{\eta_1}.
\end{align*}
Plugging these in the second $J$-equation and exploiting the first $J$-equation, we can write
\begin{equation}\label{e:finale}
12 = 25|\xi_3|^2-13|\eta_3|^2 = 25(\eta_1^2 -5\xi_1^2) - 128 + 240x\frac{\xi_1}{\eta_1} = -228 + 240x\frac{\xi_1}{\eta_1},
\end{equation}
from which we get $\eta_1 = x\xi_1$, so that $x > 0$ by \eqref{e:sign}.\ From this, \eqref{e:xy}-\eqref{e:sign} yields $y = 0$, and hence $x = \eta_1$ by \eqref{e:n}-\eqref{e:sign}.\ But then \eqref{e:finale} shows $\xi_1 = 1$.\ From the first $J$-equation and $\eqref{e:sign}$, we find also $\eta_1 = 1$, and hence $\eta_2 = 1$.\ Now \eqref{e:xi2}-\eqref{e:xi3}-\eqref{e:eta3} give $\xi_2 = \xi_3 = \eta_3 = 1$ as well, so that $A = \varphi(1,1) \in K$ and we conclude our proof.
\end{proof}

\subsection{Nondegeneracy of the polyaffine relations}\label{subsec:nondegpoly}

\begin{lemma}[Non-degeneracy]\label{lem:transversality}
For every $(\xi,\eta)\in\T^2$ we have $       \ker \Phi'(\varphi(\xi,\eta))=T_{\varphi(\xi,\eta)}K.$
\end{lemma}

\begin{proof}
As $K = \{\varphi(\xi,\eta): \xi,\eta \in \mathbb S^1\}$, see Lemma \ref{lemma:defK}, we can obtain a basis for its tangent plane by differentiating the parametrization:
\begin{equation*}
 \frac{\d}{\d t}\bigg|_{t=0}\varphi(e^{it}\xi,\eta)
 =
 \begin{bmatrix}
 i c_1\xi&0\\
 2i c_2\xi^2&i d_2\xi\eta\\
 2i c_3\xi^2\eta&i d_3\xi\eta^2
 \end{bmatrix},\;
\frac{\d}{\d t}\bigg|_{t=0}\varphi(\xi,e^{it}\eta)
 =
 \begin{bmatrix}
 	0&i d_1\eta\\
 	0&i d_2\xi\eta\\
 	i c_3\xi^2\eta&2i d_3\xi\eta^2
 \end{bmatrix}.
\end{equation*}
These are linearly independent matrices spanning $T_{\varphi(\xi,\eta)}K$, and belonging to $\ker\Phi'(\varphi(\xi,\eta))$, since $\Phi(\varphi(\xi,\eta)) \equiv 0$.\ It remains to check that there are no further kernel directions. This lengthy calculation is better performed with the help of a Computer Algebra Software. First, by equivariance of $\mb T^2$, it is enough to check the nondegeneracy at $(\xi,\eta)=(1,1)$.\ As $T_{\varphi(1,1)}K \subset \ker \Phi'(\varphi(1,1))$, our goal is to show that the rank of $\Phi'(\varphi(1,1))$ is 10.\ Let $A$ be as in \eqref{eq:defA}.\ The linearization of the constraints at $\varphi(1,1)$ is:
\begin{align*}
\Phi_{j}'(\varphi(1,1))[A]&=2\Re(\ol{c_j}a_j-\ol{d_j}b_j),\; j =1,3,\\
\Phi_{H}'(\varphi(1,1))[A]&=d_2a_1+c_1b_2-c_2b_1-d_1a_2
        -\left(\frac{3+4i}{25}\right)a_3,\\
\Phi_{{12}}'(\varphi(1,1))[A]&=\ol{c_1}a_2+c_2\ol{a_1}
        -\ol{d_1}b_2-d_2\ol{b_1} -a_1,\\
\Phi_{{13}}'(\varphi(1,1))[A]&=\ol{c_1}a_3+c_3\ol{a_1}
        -\ol{d_1}b_3-d_3\ol{b_1} +2ib_2,\\
\Phi_{{23}}'(\varphi(1,1))[A]&=\ol{c_2}a_3+c_3\ol{a_2}
        -\ol{d_2}b_3-d_3\ol{b_2}
        +(3+4i)b_1 .
\end{align*}
We parametrize $a_j = x_j + iy_j$, $b_j = u_j + iv_j$.\ Next, we differentiate the 10 functions $J_1$, $J_3$, $\Re(\Phi_H)$, $\Im(\Phi_H)$, $\Re(\Phi_{12})$, $\Im(\Phi_{12})$, $\Re(\Phi_{13})$, $\Im(\Phi_{13})$, $\Re(\Phi_{23})$, $\Im(\Phi_{23})$ with respect to the 12 variables $x_1,x_2,x_3,y_1,y_2,y_3,u_1,u_2,u_3,v_1,v_2,v_3$ to obtain $\Phi'(\varphi(1,1))$:
\[
 \Phi'(\varphi(1,1)) = \left[\begin{array}{cccccccccccc}
	2 & 0 & 0 & 1 & 0 & 0 & 0 & 0 & 0 & -1 & 0 & 0 
	\\
	0 & 0 & 4 & 0 & 0 & 3 & 0 & 0 & 2 & 0 & 0 & -3 
	\\
	0 & 0 & -\frac{3}{25} & -1 & \frac{1}{2} & \frac{4}{25} & -1 & 1 & 0 & 1 & -\frac{1}{2} & 0 
	\\
	1 & -\frac{1}{2} & -\frac{4}{25} & 0 & 0 & -\frac{3}{25} & -1 & \frac{1}{2} & 0 & -1 & 1 & 0 
	\\
	0 & 1 & 0 & 1 & \frac{1}{2} & 0 & 0 & 0 & 0 & -1 & -\frac{1}{2} & 0 
	\\
	1 & -\frac{1}{2} & 0 & -2 & 1 & 0 & -1 & \frac{1}{2} & 0 & 0 & 0 & 0 
	\\
	2 & 0 & 1 & \frac{3}{2} & 0 & \frac{1}{2} & 1 & 0 & 0 & -\frac{3}{2} & -2 & -\frac{1}{2} 
	\\
	\frac{3}{2} & 0 & -\frac{1}{2} & -2 & 0 & 1 & -\frac{3}{2} & 2 & \frac{1}{2} & -1 & 0 & 0 
	\\
	0 & 2 & 1 & 0 & \frac{3}{2} & 1 & 3 & 1 & 0 & -4 & -\frac{3}{2} & -1 
	\\
	0 & \frac{3}{2} & -1 & 0 & -2 & 1 & 4 & -\frac{3}{2} & 1 & 3 & -1 & 0 
\end{array}\right] \in \R^{10\times 12}.
\]
The rank of $\Phi'(\varphi(1,1))$ is $10$, as wanted: the Maple code, where the differential is called \texttt{DIFF}, used to make this check is attached to the ArXiv version of this paper.\ 
\end{proof}

\section{Strongly polyconvex extensions of the squared distance}
\label{sec:squared-distance-extension}

Aim of this section is to show Theorem \ref{thm:main}, i.e. to build a strong polyconvex extension of the $\dist_K^2$ function.\ Some results we will consider are rather general, hence we will work in $\R^{m\times n}$, $m,n \ge 2$, even if we are only interested in the case $n = 2$.\ We will use freely the notation of Definition \ref{def:pc}.\ In addition, we introduce $P \in \R^{(nm)\times \tau(m,n)}$, the projection matrix which acts on $M(X) \in \R^{\tau(m,n)}$ as $PM(X) \equiv X$, $\forall X \in \R^{m\times n},$ so that $P$ is the projection onto suitable $nm$ coordinates of $\R^{\tau(m,n)}$.\ We also let $P^T$ denote the transpose projection. 

We start with the following simple extension result:

\begin{lemma}
	\label{lem:convex-extension-minors}
	Let $U'\Subset U\subset\mathbb R^{m\times n}$, where $U$ is bounded.\ Suppose that $h\in C^\infty(\overline U)$ and that	$X\mapsto c_X\in\mathbb R^{\tau(m,n)}$ is smooth.\ If, for some
	$\kappa>0$, we have
	\begin{equation}	\label{eq:uniform-polyaffine-support-gap}
		h(Y)-h(X)-\bigl\langle c_X,M(Y)-M(X)\bigr\rangle \geq\kappa|Y-X|^2 \qquad\forall X,Y\in\overline U,
	\end{equation}
	then there is a smooth convex function
	$G\colon\mathbb R^{\tau(m,n)}\to\mathbb R$ such that $G(M(X))=h(X),$ $\forall X\in U'$. Moreover,
	\begin{equation}\label{eq:growth}
	|G(\mc M)| \le C(1 + |\mc M|), \quad \forall \mc M \in \R^{\tau(m,n)}
	\end{equation}
\end{lemma}

\begin{proof}
	Set $\kappa' \equiv \kappa/\Lip_R^2(M)$, where $\Lip_R(M)$ is the Lipschitz constant of $M$ over a ball $B_R \subset \R^{m\times n}$ with $U \subset B_{R/2}$, and define:
	\[
	\varphi(X,\mathcal{M})\equiv h(X)+\bigl\langle c_X,\mc M-M(X)\bigr\rangle +\frac{\kappa'}{2}\bigl|\mc M-M(X)\bigr|^2 \text{ and } \widehat G(\mc M)\equiv 	\sup_{X\in\overline U}\varphi(X,\mc M).
	\]
	Then, $\widehat G\colon \R^{\tau(m,n)}\to \R$ is convex, being the supremum of convex functions.\ If $Y\in U$,
	\begin{equation}\label{eq:XMY}
		\varphi(X,M(Y)) \overset{\eqref{eq:uniform-polyaffine-support-gap}}{\leq} h(Y)- \frac{\kappa}{2}|Y-X|^2	=\varphi(Y,M(Y))- \frac{\kappa}{2}|Y-X|^2.
	\end{equation}
	Thus $X=Y$ is the unique maximizer defining $\widehat G(M(Y))$, and
	$\widehat G(M(Y))=h(Y)$.\ Since it is an interior maximizer, $\D_X\varphi(Y,M(Y))=0$.\ From this, and rewriting \eqref{eq:XMY} for $X=Y+tZ$ as
	\[
	\varphi(Y+tZ,M(Y)) - \varphi(Y,M(Y)) \le - \frac{t^2\kappa}{2}|Z|^2,
	\]
	we deduce that:
	\begin{equation}\label{eq:nonde}
	\D^2_{XX}\varphi(Y,M(Y))[Z,Z]\leq-\kappa|Z|^2, \quad \forall Z \in \R^{m\times n}.
    \end{equation}
	Now let $X \in \overline{U}$ be any maximizer for $\widehat G (\mc M)$, for any $\mc M \in \R^{\tau(m,n)}$.\ Take $T \in \overline{U'}$ with $|\mc M - M(T)| = \dist(\mc M,M(\overline{U'}))$ and write:
	\begin{align*}
	\widehat G(\mc M) - \widehat G(M(T)) &= \varphi(X,\mc M) - h(T) = \varphi(X,\mc M) - \varphi(X,M(T)) + \varphi(X,M(T)) - h(T) \\
	&\overset{\eqref{eq:XMY}}{\le } \varphi(X,\mc M) - \varphi(X,M(T)) -\frac{\kappa}{2}|X-T|^2.
	\end{align*}
	 Since $\widehat G$ is finite and convex, and thus Lipschitz on compact sets, we readily deduce from the previous inequality that $$|X-T| \le C(|\mathcal{M}|)\sqrt{\dist(\mc M,M(\overline{U'}))}.$$ Therefore, for all $\mc M$ near $M(\overline{U'})$, every maximizer defining $\widehat G$ also lies close to a single matrix $T$ chosen as above.\ In particular, every maximizer lies in $U$.\ Now \eqref{eq:nonde} shows, through the implicit function theorem applied to the map $D_X\varphi(X,\mc M)$, that the maximizer is unique and depends smoothly on $\mc M$ near $M(\overline{U'})$.\ Hence $\widehat G$ is smooth there, and it remains to extend it to a smooth convex function on the whole space.\ To this end, we start by noticing that we can expand the square in the definition of $\widehat G$ to write:
	\[
	\widehat G(\mc M)= \frac{\kappa'}{2}|\mc M|^2+\ell(\mc M),
	\]
	where $\ell$ is convex and smooth in an open, bounded set $O$ with $M(\overline{U'}) \subset O$.\ Take now open $O_i$ with 
	\begin{equation}\label{eq:conta}
	M(\overline{U'})\subset O_3 \Subset O_2 \Subset O_1 \Subset O.
	\end{equation}
	As $\widehat G$ is globally convex, it is convex on $\overline{O}$ in the sense of \cite{Yan2014}.\ Moreover, its Hessian is positive definite in $O$ and hence we can extend it to a smooth, convex function $G_1$ in the convex hull of $O_1$, $\co O_1$, by \cite[Theorem 3.4]{Yan2014}.\ Up to considering a slightly smaller $O_1$ still fulfilling \eqref{eq:conta}, we can assume that $G_1$ is Lipschitz over $\overline{\co O_1}$, and hence by \cite[Theorem 4.1]{Yan2014}, it admits a convex extension $G_2$ to the full space.\ It is not hard to see from \cite[Theorem 4.1]{Yan2014} that $G_2$ is still globally $L$-Lipschitz, for some $L >0$.\ The function $G_2$ is only introduced as a technical step towards the smooth extension $G_1$ outside $\co O_3$.\ Indeed, we now consider the convex function $m(a,b) \equiv \max\{a,b\}$, and mollify it with a radial kernel in such a way that the resulting function $\tilde m(a,b)$ coincides with $m(a,b)$ in $\{(a,b) \in \R^2: |a-b|\ge  1\}$, and that it is still convex and nondecreasing in each variable.\ Similarly, we consider the function $\delta(\mc M)$, obtained by mollifying $\dist(\mc M, \co O_2)$, so that:
	\begin{enumerate}
		\item\label{prop:1} $\delta$ is convex and smooth;
		\item\label{prop:2} $\delta \equiv 0$ on $\co O_3$ and $\delta(\mc M) \ge \eta > 0$ in $(\co O_1)^c$;
		\item\label{prop:3} $\delta \ge \dist(\cdot, \co O_2)$ everywhere.
	\end{enumerate}
	We can finally set, for $\lambda = L + (\sup_{\co O_2}G_2 - \inf_{\co O_3}G_2 + 10)/\eta$:
	\[
	G(\mc M) \equiv \tilde m(G_2(\mc M), \lambda\delta(\mc M)+ \inf_{\co O_3} G_2-2).
	\]
	$G$ is still convex, thanks to the properties of $\tilde m$ and the convexity of the functions in the arguments.\ In $\co O_3$ we have $ G(\mc M) = G_2(\mc M)$, and hence $ G(\mc M) = \widehat G(\mc M)$ in $O_3$.\ If we now consider any $\mc M \in (\co O_1)^c$ and $\mc M_0 \in \overline{\co O_2}$ with $|\mc M-\mc M_0| = \dist(\mc M,\co O_2)$, we can bound, using that $G_2$ is $L$-Lipschitz:
	\[
	G_2(\mc M) \le L\dist(\mc M, \co O_2) + \sup_{\co O_2} G_2 \overset{\ref{prop:3}}{\le} L\delta(\mc M) + \sup_{\co O_2} G_2 \overset{\ref{prop:2}}{<} \lambda\delta(\mc M) + \inf_{\co O_3}G_2  - 10,
	\]
	using our choice of $\lambda$.\ This shows that $G(\mc M) = \lambda\delta(\mc M)+ \inf_{\co O_3} G_2-2$ in $(\co O_1)^c$.\ Hence, $G$ is the required smooth, convex extension of $\widehat G$ to the whole $\R^{\tau(m,n)}$, with $G = \widehat G$ on $O_3$, and hence $G(M(X)) = h(X)$ for all $X \in U'$.\ Finally, since $G$ coincides with $ \lambda\delta(\mc M)+ \inf_{\co O_3} G_2-2$ outside a large ball, \eqref{eq:growth} is also fulfilled.\ This concludes the proof.
\end{proof}

With this auxiliary lemma at our disposal, we can show the following general result:

\begin{proposition}
\label{prop:abstract-extension}
Let $K\subset\mathbb R^{m\times n}$ be a compact smoothly embedded
submanifold. Suppose that there are $\delta>0$ and a smooth family
$K\ni A\mapsto\Psi_A$ of polyaffine functions such that
\begin{equation}
\label{eq:abstract-hypothesis}
\Psi_A=0\quad\text{on }K,
\qquad
\Psi_A'(A)[A-B]\geq\delta|A-B|^2
\quad\text{for every }A,B\in K.
\end{equation}
Then there are $\rho,\varepsilon>0$ and a smooth convex function
$G\colon\mathbb R^{\tau(m,n)}\to\mathbb R$ such that the integrand
\begin{equation*}
F(A)\equiv\varepsilon|A|^2+G(M(A)),
\end{equation*}
satisfies $F(A)=\dist^2_K(A)$ whenever $\dist_K(A)<\rho$, and
\begin{equation}\label{eq:growth2}
|F(A)| \le C(1 + |A|^{\min\{m,n\}}), \quad \forall A \in \R^{m\times n}.
\end{equation}
\end{proposition}

\begin{proof}
Write $d(X)\equiv \dist^2_K(X)$.\ For $B\in K$, $T\in T_BK$, and $N\perp T_BK$,
\begin{equation}
\label{eq:squared-distance-Hessian}
d'(B)=0,
\qquad
d''(B)[T+N,T+N]=2|N|^2.
\end{equation}
Since the family $A\mapsto\Psi_A$ is smooth and $K$ is compact, we can consider a constant $C \ge 2\delta$ to get, through Young's inequality:
\begin{equation}
\label{eq:boundsPsi''}
2|\Psi_A''(X)[T,N]|+|\Psi_A''(X)[N,N]|
\leq\frac\delta2|T|^2+\frac C2|N|^2,
\end{equation}
for every $A\in K$, $X$ in a neighborhood of $K$, $T\in T_AK$, and $N\perp T_AK$. Choose
\begin{equation}\label{eq:choices}
\lambda>0\quad\text{with}\quad\lambda C\leq1,
\qquad
\varepsilon\equiv\frac{\lambda\delta}{2},
\qquad
h(X)\equiv d(X)-\varepsilon|X|^2.
\end{equation}
To conclude the proof we only need to apply Lemma \ref{lem:convex-extension-minors}, and in particular show that $h$ solves \eqref{eq:uniform-polyaffine-support-gap} for some vectors $c_X$.\ Indeed, once that is done, the function $G$ provided by Lemma \ref{lem:convex-extension-minors} is such that $G(M(A)) = h(A) = d(A) - \varepsilon |A|^2$ for $A$ in a neighborhood of $K$, and \eqref{eq:growth2} is a consequence of \eqref{eq:growth} and the definition of $M$.\ To define the vectors $c_X$, let $\pi$ be the nearest-point projection onto $K$, defined in a tubular neighborhood of $K$.\ For $X$ in that neighborhood, let $B=\pi(X)$ and write $\Psi_B(Y)=a_B+\langle b_B,M(Y)\rangle$, recalling Lemma \ref{lem:polyaff}.\ Finally, set:
\[
c_X\in\mathbb R^{\tau(m,n)},
\qquad
c_X\equiv P^T\bigl(h'(X)+\lambda\Psi_B'(X)\bigr)-\lambda b_B.
\]
We start by noticing that, with this choice,
\begin{equation}
	\label{eq:polyaffine-support-minors}
	\bigl\langle c_X,M(Y)-M(X)\bigr\rangle =h'(X)[Y-X]-\lambda\bigl(\Psi_B(Y)-\Psi_B(X)-\Psi_B'(X)[Y-X]\bigr).
\end{equation}
Inequality \eqref{eq:uniform-polyaffine-support-gap} follows from a simple contradiction-compactness argument based on the following two properties:
\begin{enumerate}[label=(\roman*)]
\item \label{eq:glo} \eqref{eq:uniform-polyaffine-support-gap} holds for all $X,Y \in K$;
\item \label{eq:lo} there exists $\kappa > 0$ such that 
\begin{equation}
	\label{eq:modified-Hessian-positive}
	\bigl(h''(A)+\lambda\Psi_A''(A)\bigr)[Z,Z]\geq\kappa|Z|^2 \qquad\forall A\in K.
\end{equation}
\end{enumerate}
 To establish \ref{eq:glo}-\ref{eq:lo}, we start by observing that, for $A,B\in K$, \eqref{eq:abstract-hypothesis}-\eqref{eq:squared-distance-Hessian} give:
\[
h(A)=-\varepsilon|A|^2,
\qquad h'(A)=-2\varepsilon A,
\qquad \Psi_A(A)=\Psi_A(B)=0.
\]
Substitution in \eqref{eq:polyaffine-support-minors}, followed by
\eqref{eq:abstract-hypothesis} and the choice
$\varepsilon=\lambda\delta/2$, yields
\begin{equation}
\label{eq:polyaffine-support-on-K}
\begin{aligned}
h(B)-h(A)-\bigl\langle c_A,M(B)-M(A)\bigr\rangle
&=-\varepsilon|A-B|^2+\lambda\Psi_A'(A)[A-B]\geq\varepsilon|A-B|^2,
\end{aligned}
\end{equation}
which shows \ref{eq:glo}.\ To see \ref{eq:lo}, start by dividing \eqref{eq:polyaffine-support-on-K} by $|A-B|^2$ and let $B\to A$ along a curve in $K$ with velocity $T\in T_AK$:
\[
\bigl(h''(A)+\lambda\Psi_A''(A)\bigr)[T,T]
\geq2\varepsilon|T|^2=\lambda\delta|T|^2.
\]
For $Z=T+N$, where $T\in T_AK$ and $N\perp T_AK$, it follows
from \eqref{eq:squared-distance-Hessian}--\eqref{eq:boundsPsi''} that
\begin{align*}
\bigl(h''(A)+\lambda\Psi_A''(A)\bigr)[Z,Z]
&\geq\frac{\lambda\delta}{2}|T|^2
 +\left(2-\lambda\delta-\frac{\lambda C}{2}\right)|N|^2\geq\frac{\lambda\delta}{2}|T|^2+|N|^2.
\end{align*}
Here the last inequality follows from \eqref{eq:choices}.\ This shows \ref{eq:lo} with $\kappa\equiv \lambda\delta/2 < 1$.
\end{proof}

When applying Proposition \ref{prop:abstract-extension}, we want the polyconvex integrand $F$ to vanish exactly on $K$. The next lemma gives a simple necessary and sufficient condition for this to  hold.

\begin{lemma}
\label{lem:global-majorant}
Under the hypotheses of Proposition~\ref{prop:abstract-extension},
suppose in addition that
\begin{equation}
\label{eq:zerosetcond}
\operatorname{conv}M(K)\cap M(\mathbb R^{m\times n})=M(K).
\end{equation}
Then the integrand in the proposition may be chosen so that
\begin{equation}
\label{eq:majorant}
F(X)\geq\dist^2_K(X)\qquad\forall X\in\mathbb R^{m\times n}.
\end{equation}
In particular, $F\geq0$ and $F^{-1}(0)=K$.\ If $\min\{m,n\} = 2$, then we can achieve $|F''|\leq C$. 
\end{lemma}

\begin{proof}
Let $F_0(X)=\varepsilon|X|^2+G_0(M(X))$ be given by
Proposition~\ref{prop:abstract-extension}, and let $U_0$ be a tubular
neighborhood of $K$ on which $F_0=\dist^2_K$.\ The desired integrand will be of the form
\begin{align*}
F_1(X) & \equiv F_0(X) + \lambda f(M(X))\\
& = \e |X|^2 + G_0 (M(X)) + \lambda f(M(X))  \equiv \e |X|^2 + G_1(M(X)),
\end{align*}
where $f\geq0$ is smooth, convex, vanishes near $M(K)$, and
$\lambda$ is chosen so that $\lambda f(M(X))$ dominates
$\dist^2_K(X)-F_0(X)$ away from $U_0$. To achieve this, let first $K_0\equiv\operatorname{conv}M(K)$ and choose $\sigma>0$ such that
\begin{equation}
\label{eq:convex-hull-neighborhood}
M^{-1}(K_0+\overline{\mb B}_{4\sigma})\subset U_0.
\end{equation}
This is possible by \eqref{eq:zerosetcond} and compactness.\ Let next $\eta_\sigma$ be a standard radial mollifier with $\supp \eta_\sigma \subset \mb B_\sigma$, and set
\[
g\equiv \eta_\sigma*\dist(\cdot,K_0+\overline{\mb B}_{2\sigma})\geq 0.
\]
Then $g$ is convex, globally Lipschitz, and vanishes exactly on $K_0+\overline{\mb B}_\sigma$. Choose
$R_0$ so that $P(K_0+\overline{\mb B}_\sigma)\subset
\overline{\mb B}_{R_0}$, and let
$\chi\colon\mathbb R\to[0,\infty)$ be smooth, convex, and
nondecreasing, with $\chi=0$ on $( -\infty,0]$, $\chi>0$ on
$(0,\infty)$, and $\chi$ affine on $[1,\infty)$. Thus
\begin{equation}\label{eq:ff}
f(\mc M)\equiv g(\mc M)+\chi\bigl(|P\mc M|^{\min\{m,n\}}-R_0^{\min\{m,n\}}\bigr)
\end{equation}
is smooth, convex, nonnegative and vanishes precisely on $K_0+\overline{\mb B}_\sigma$.\ We now choose the constant $\lambda$. By \eqref{eq:convex-hull-neighborhood}, $f(M(X))\ge \delta >0$ whenever
$X\notin U_0$, and by \eqref{eq:ff}, $f(M(X)) \ge c(|X|^{\min\{m,n\}})- C$ for large values of $X$. Hence, using also \eqref{eq:growth2} we see that
\[
\lambda_*\equiv
\sup_{X\notin U_0}
\frac{\dist^2_K(X)-F_0(X)}{f(M(X))}<\infty.
\]
For $\lambda>\max\{0,\lambda_*\}$, we then have \eqref{eq:majorant}.

We now assume $\min\{m,n\}=2$, and we modify $F_1$ so that its Hessian is bounded.\ By \eqref{eq:growth2}-\eqref{eq:ff} we have that $|G_1(M(X))|\leq C(1+|X|^2)$.\ Similarly to the last steps of Lemma \ref{lem:convex-extension-minors}, we can take a smooth convex maximum of $G_1(\mc M)$ and $L|P\mc M|^2-C_1$, with $L$ and then $C_1$ sufficiently large to obtain a smooth convex function $G$ which agrees with $G_1$ near $M(K)$, dominates $G_1$ everywhere, and satisfies $G(M(X))=L|X|^2-C_1$ whenever $|X|$ is sufficiently large.\ Thus $F(X)=\varepsilon|X|^2+G(M(X))$ still satisfies \eqref{eq:majorant}, but it is also quadratic outside a compact set, thus $|F''|\leq C$.
\end{proof}

Finally, we give a simple general criterion for the conditions in Lemma \ref{lem:global-majorant} to hold.

\begin{lemma}
\label{lem:spherical-transversality-chord}
Let $K\subset\mathbb R^{m\times n}$ be a compact smoothly embedded submanifold with $$K\subseteq  \{A\in \R^{m\times n} :|A|=c\}, \qquad c>0.$$
With $\Phi\equiv (\Phi_1,\ldots,\Phi_N)$, where each $\Phi_j\colon \R^{m\times n}\to \R$ is polyaffine,
suppose also that
\begin{equation}
\label{eq:abstractcondsK}
K=\Phi^{-1}(0),
\qquad
\ker\Phi'(A)=T_AK\qquad\forall A\in K.
\end{equation}
Then there is a smooth family $K\ni A\mapsto\Psi_A$ of linear
combinations of the $\Phi_j$ such that
\[
\Psi_A'(A)[A-B]=|A-B|^2
\qquad\forall A,B\in K.
\]
Moreover, we also have
\[
\operatorname{conv}M(K)\cap M(\mathbb R^{m\times n})=M(K).
\]
\end{lemma}

\begin{proof}
Since $K\subset\{|A|=c\}$, we have
$A\perp T_AK=\ker\Phi'(A)$, and hence
$2A\in\operatorname{im}\Phi'(A)^T$. Since $K$ is smooth and $T_AK =\ker \Phi'(A)$, the maps $\Phi'(A)$ have
constant rank, so their Moore--Penrose inverses $\Phi'(A)^\dagger$ depend smoothly on
$A$. Thus
\[
A\mapsto q_A, \qquad q_A \equiv2\bigl(\Phi'(A)^\dagger\bigr)^T A\in\mathbb R^N
\]
is smooth and satisfies $\Phi'(A)^Tq_A=2A$. Define
$\Psi_A(X)=\langle q_A,\Phi(X)\rangle$ for all $X \in \R^{m\times n}$. This is a smooth family of
polyaffine functions and $\Psi_A=0$ on $K$. For
$A,B\in K$ we have
\[
\Psi_A'(A)[A-B]
=\langle2A,A-B\rangle
=|A-B|^2,
\]
where the last identity follows from $|A|=|B|=c$. Finally, suppose that $M(X)\in\operatorname{conv}M(K)$.\ Write $\Phi_j(X) =\tilde \Phi_j(M(X))$, for $\tilde \Phi_j(\mc M) = a_j + (b_j,\mc M)$ for $\mc M \in \R^{\tau(m,n)}$, thanks to Lemma \ref{lem:polyaff}.\ As $M(X)\in\operatorname{conv}M(K)$ and $\Phi_j(A) = \tilde{\Phi}_j(M(A)) = 0$ for all $A \in K$, we deduce $\tilde{\Phi}_j(M(X))= 0$, for all $j$.\ Thus $\Phi(X)=0$ and by assumption $X\in K$. The reverse inclusion is immediate.
\end{proof}

\begin{proof}[Proof of Theorem \ref{thm:main}]
For $A=\varphi(\xi,\eta)\in K$, Lemma \ref{lemma:defK} shows that $|A|^2=28$.
Moreover, Lemmas~\ref{lem:zero-set} and~\ref{lem:transversality} show that \eqref{eq:abstractcondsK} holds, 
so Lemma~\ref{lem:spherical-transversality-chord} shows that we can apply Lemma~\ref{lem:global-majorant} to find  a smooth strongly polyconvex
integrand $F$, with $\sup|F''|<\infty$, such that
\[
F\geq\dist^2_K\geq0,
\qquad
F^{-1}(0)=K.
\]
Since $\D u\in K$ almost everywhere in $\mathbb B^2$,
\[
\int_{\mathbb B^2}F(\D u)\,\d x=0
\leq\int_{\mathbb B^2}F(\D v)\,\d x
\qquad\forall v\in u+W^{1,2}_0(\mathbb B^2,\mathbb R^6).
\]
Suppose that equality holds. Then $F(\D v)=0$ almost everywhere, so
$\D v\in K$ almost everywhere. By Theorem~\ref{thm:rigid} below, since $u=v$ on $\mb S^1$ it is easy to see that we necessarily have $u=v$ in $\mb B^2$.
\end{proof}

\section{Rigidity of the inclusion: proof of Theorem \ref{thm:rigid}}

It is, once again, convenient to use complex notation.\ To this end, let us recall the chain rules for Wirtinger derivatives:
\begin{gather}
	\label{eq:chain}
	(g\circ f)_w = g_z \circ f f_w + g_{\bar z} \circ f \overline{f_{\bar w}}, \qquad
	(g\circ f)_{\bar w} = g_{z} \circ f f_{\bar w} + g_{\bar z} \circ f \overline{f_{w}},\\
	\label{eq:invs}
	(g^{-1})_w  = \frac{\overline{g_z}\circ g^{-1}}{\det \D g \circ g^{-1}}, \qquad (g^{-1})_{\bar w} = -\frac{g_{\bar z}\circ g^{-1}}{\det \D g\circ g^{-1}}.
\end{gather}

Now we can show Theorem \ref{thm:rigid}.\ The inclusion $\D v\in K$ a.e. means that there are measurable fields $p,q\colon \Omega\to \mb S^1$ such that
\begin{align}
\label{eq:system}
\begin{cases}
 v^1_z = c_1 p,  \\
 v^1_{\bar z}  = d_1 q,
 \end{cases}
 \qquad
 \begin{cases}
v^2_z= c_2 p^2, \\v^2_{\bar z}  = d_2 pq,
\end{cases}
\qquad
\begin{cases}
v^3_z = c_3 p^2 q, \\
 v^3_{\bar z}  =  d_3 p q^2.
 \end{cases}
\end{align}

The map $v^1$ is Lipschitz and has Jacobian constantly equal to 1 by Lemma \ref{lem:zero-set}.\ Thus, it is quasiregular and by \cite[Section 5.5]{Astala2009}, it admits a discrete branch set, namely it is locally invertible outside a discrete set.\ We will work in a small ball outside the branch set, and the global statement can be inferred by the connectedness of the complement of the branch set. Take then such a small ball in which it is invertible, and consider $f\equiv (v^{1})^{-1}$, which is again a Lipschitz map.\ Let us write 
$$P\equiv p\circ f, \qquad Q\equiv q\circ f, \qquad S\equiv PQ.$$ Since $\det \D v^1=1$, from \eqref{eq:invs}--\eqref{eq:system} we have
\begin{equation}
\label{eq:f}
f_w = \overline{c_1} \overline P, \qquad f_{\bar w} = -d_1 Q.
\end{equation}
Using \eqref{eq:chain}--\eqref{eq:system}, we now compute the equations for $V^2 \equiv v^2\circ f, V^3\equiv v^3 \circ f$, thanks to the chain rule for quasiregular maps \cite[Lemma 9.6]{Bojarski1983}. First, 
\begin{equation}
\label{eq:V2}
V^2_w = (\overline{c_1} c_2 -\overline{d_1} d_2)P = c_1 P, \qquad V^2_{\bar w} = (c_1 d_2 -c_2d_1) P S = d_1PS.
\end{equation}
Notice that \eqref{eq:V2} can be rewritten as a Beltrami equation for $V^2$:
\begin{equation}
\label{eq:V2belt}
V^2_{\bar w} = \Big(\frac{d_1}{c_1} S\Big) V^2_w.
\end{equation}
Similarly to \eqref{eq:V2}, we derive the equations for $V^3$:
\begin{equation}
\label{eq:V3}
V^3_w = (\overline{c_1} c_3 -\overline{d_1} d_3)S = 2 S, \qquad V^3_{\bar w} = (c_1 d_3 -d_1 c_3)S^2= -S^2.
\end{equation}
Notice that $|S|=1$: thus $V^3$ solves the differential inclusion $$\D V^3 \in \Gamma\equiv \{(2\eta,-\eta^2):\eta\in \mb S^1\}.$$
Now, $\Gamma$ is a degenerate elliptic curve in the sense of \cite{Lamy2024}.\ Indeed, if $\eta,\xi\in \mb S^1$ then $$|2 \eta-2\xi|^2 -|\eta^2-\xi^2|^2 = |\eta-\xi|^2(4-|\eta+\xi|^2) = |\eta-\xi|^4,$$
since $|\eta+\xi|^2 + |\eta-\xi|^2 = 4$ as $\eta, \xi\in \mb S^1$.\ On the other hand, $$|(2\eta,-\eta^2)-(2\xi,-\xi^2)|^2 = 4 |\eta-\xi|^2+ |\eta^2-\xi^2|^2 = |\eta-\xi|^2 (4+ |\eta+\xi|^2) \leq 8 |\eta-\xi|^2$$ so indeed $$\det(A-B)\geq c |A-B|^4 \quad \text{for all } A,B\in \Gamma.$$
It follows from the beautiful result \cite[Theorem 1.3]{Lamy2024} that $V^3$ is locally $W^{2,\infty}$ away from a discrete set $D$.\ Away from $D$, we see from \eqref{eq:V3} that $S$ is $W^{1,\infty}_\loc$, and so by \eqref{eq:V2belt} and standard elliptic regularity we have $V^2\in W^{2,p}_\loc$ for all $p<\infty$.\ Thus, from \eqref{eq:V2}, also $P\in W^{1,p}_\loc$ and similarly for $Q$.\ Since our fields have Sobolev regularity away from a discrete set, we can apply the compatibility condition $\p_{w \bar w}=\p_{\bar w w}$ and the chain rule to \eqref{eq:f}, \eqref{eq:V2} and \eqref{eq:V3} to get
\begin{equation}
\label{eq:chainrule}
\begin{cases}
c_1 P_w  =d_1 (P\bar S)_{\bar w},\\
c_1 P_{\bar w} = d_1 (PS)_w,\\
S_{\bar w} = -S S_w.
\end{cases}
\end{equation}
It is useful to write $P=e^{i \rho}$ and $S=e^{i \sigma}$, so that $P_w = i \rho_w P, P_{\bar w} = i \rho_{\bar w} P$ and similarly for $S$. Then \eqref{eq:chainrule} reduces to:
\begin{equation*}
	\label{eq:chainrule2}
	\begin{cases}
c_1 S\rho_w=d_1(\rho_{\bar w}+S\sigma_w),\\
c_1\rho_{\bar w}=d_1 S(\rho_w+\sigma_w),\\
\sigma_{\bar w} = -S \sigma_w.
\end{cases}
\end{equation*}
Subtracting the first two equations gives $\rho_{\bar w}=S\rho_w$, since $c_1+d_1\neq0$. Substituting this back into the second equation and using $c_1-d_1=1$, we obtain
$$
\rho_{\bar w}=d_1 S\sigma_w=\frac{iS}{2}\sigma_w.
$$
Combining these results, we finally get the identity:
\begin{equation}
\label{eq:simplesys}
 \rho_w = \frac i 2 \sigma_w.
\end{equation}
Since $\rho$ and $\sigma$ are real-valued, we then deduce that $h\equiv \rho + \frac i 2 \sigma$ is holomorphic.\ To conclude, we return to the third equation in \eqref{eq:chainrule}, noting that $S=e^{i \sigma} = e^{h-\bar h}$, thus
$$-\overline{h_w} e^{h-\bar h} =- h_w e^{2(h-\bar h)} \quad \implies \quad
h_w e^h = \overline{h_w e^h}.$$
Hence the holomorphic function $(e^h)_w$ is real-valued, thus constant. If this constant is zero, then $P$ and $S$ are constant, so $\D v$ is constant. We may therefore assume that this constant is non-zero and write
$$e^{h(w)}=a(w-w_0), \qquad a \in \R\setminus\{0\}, \qquad w_0 \in \C,$$
thus 
$$S(w)= \frac{w-w_0}{\overline{w-w_0}} = e^{2 i \theta},$$
using polar coordinates centered at $w_0$, i.e. $w-w_0=r e^{i \theta}$. Thus $\sigma(w)=2\theta$ and from this and the equality $a(w-w_0) = e^h = e^{\rho + \frac{i}{2}\sigma}$, we get
$$P(w) = e^{i\rho} = \lambda |w-w_0|^i, \qquad \lambda\in \mb S^1.$$
Finally, we can simply integrate \eqref{eq:f}, \eqref{eq:V2}, and \eqref{eq:V3} to get that, up to additive constants,
\begin{align*}
f(w) & = z_0+\bar \lambda (w-w_0) |w-w_0|^{-i},\\
V^2(w)& = \lambda(w-w_0)|w-w_0|^i,\\
V^3(w)& =\frac{(w-w_0)^2}{\overline{w-w_0}},
\end{align*}
where $z_0\in \C$. For the inverse, we have
$$w= v^1(z)= w_0 + \lambda (z-z_0)|z-z_0|^i,$$
and the formulas for $V^2$ and $V^3$ give
$$\D v=\D(S_\lambda u(\cdot-z_0))$$
as wished.

\let\oldthebibliography\thebibliography
\let\endoldthebibliography\endthebibliography
\renewenvironment{thebibliography}[1]{
	\begin{oldthebibliography}{#1}
		\setlength{\itemsep}{0.5pt}
		\setlength{\parskip}{0.5pt}
	}
	{
	\end{oldthebibliography}
}

{\small
	\bibliographystyle{abbrv-andre}
	\bibliography{library}
}

\end{document}